\documentclass[a4paper, reqno, 12pt]{amsart}

\usepackage[usenames,dvipsnames]{color}
\usepackage{amsthm,amsfonts,amssymb,amsmath,amsxtra}
\usepackage[all]{xy}
\SelectTips{cm}{}
\usepackage{xr-hyper}
\usepackage{hyperref}
\usepackage{verbatim}

\usepackage[margin=1.25in]{geometry}
\usepackage{mathrsfs}

\RequirePackage{xspace}
\RequirePackage{etoolbox}
\RequirePackage{varwidth}
\RequirePackage{enumitem}
\RequirePackage{tensor}
\RequirePackage{mathtools}
\RequirePackage{longtable}
\RequirePackage{multirow}

\def\le{\leqslant}
\def\a{\alpha}

\def\e{\epsilon}

\def\s{\sigma}
\def\t{\tau}
\def\th{\theta}
\def\k{\kappa}

\def\i{^{-1}}

\def\<{\langle}
\def\>{\rangle}

\newcommand{\bG}{\mathbf G}

\newcommand{\BC}{\ensuremath{\mathbb {C}}\xspace}

\newcommand{{\BG}}{\ensuremath{\mathbb {G}}\xspace}

\newcommand{{\BK}}{\ensuremath{\mathbb {K}}\xspace}

\newcommand{\BN}{\ensuremath{\mathbb {N}}\xspace}

\newcommand{\BQ}{\ensuremath{\mathbb {Q}}\xspace}
\newcommand{\BR}{\ensuremath{\mathbb {R}}\xspace}
\newcommand{\BS}{\ensuremath{\mathbb {S}}\xspace}

\newcommand{\CI}{\ensuremath{\mathcal {I}}\xspace}

\newcommand{\CO}{\ensuremath{\mathcal {O}}\xspace}

\newcommand{\Ad}{{\mathrm{Ad}}}
\newcommand{\ad}{{\mathrm{ad}}}

\DeclareMathOperator{\Aut}{Aut}

\newcommand{\aff}{\mathrm{aff}}

\def\tW{\tilde W}

\def\kk{\mathbf k}

\newtheorem{theorem}{Theorem}
\newtheorem{proposition}[theorem]{Proposition}
\newtheorem{lemma}[theorem]{Lemma}

\theoremstyle{definition}

\newtheorem{remark}[theorem]{Remark}

\numberwithin{equation}{section}
\numberwithin{theorem}{section}

\setitemize[0]{leftmargin=*,itemsep=\the\smallskipamount}
\setenumerate[0]{leftmargin=*,itemsep=\the\smallskipamount}

\renewcommand{\to}{%
   \ifbool{@display}{\longrightarrow}{\rightarrow}%
   }
\let\shortmapsto\mapsto
\renewcommand{\mapsto}{%
   \ifbool{@display}{\longmapsto}{\shortmapsto}%
   }
\newlength{\olen}
\newlength{\ulen}
\newlength{\xlen}
\newcommand{\xra}[2][]{%
   \ifbool{@display}%
      {\settowidth{\olen}{$\overset{#2}{\longrightarrow}$}%
       \settowidth{\ulen}{$\underset{#1}{\longrightarrow}$}%
       \settowidth{\xlen}{$\xrightarrow[#1]{#2}$}%
       \ifdimgreater{\olen}{\xlen}%
          {\underset{#1}{\overset{#2}{\longrightarrow}}}%
          {\ifdimgreater{\ulen}{\xlen}%
             {\underset{#1}{\overset{#2}{\longrightarrow}}}
             {\xrightarrow[#1]{#2}}}}%
      {\xrightarrow[#1]{#2}}
   }
\makeatother
\newcommand{\xyra}[2][]{%
   \settowidth{\xlen}{$\xrightarrow[#1]{#2}$}%
   \ifbool{@display}%
      {\settowidth{\olen}{$\overset{#2}{\longrightarrow}$}%
       \settowidth{\ulen}{$\underset{#1}{\longrightarrow}$}%
       \ifdimgreater{\olen}{\xlen}%
          {\mathrel{\xymatrix@M=.12ex@C=3.2ex{\ar[r]^-{#2}_-{#1} &}}}%
          {\ifdimgreater{\ulen}{\xlen}%
             {\mathrel{\xymatrix@M=.12ex@C=3.2ex{\ar[r]^-{#2}_-{#1} &}}}
             {\mathrel{\xymatrix@M=.12ex@C=\the\xlen{\ar[r]^-{#2}_-{#1} &}}}}}%
      {\mathrel{\xymatrix@M=.12ex@C=\the\xlen{\ar[r]^-{#2}_-{#1} &}}}%
   }
\makeatletter
\newcommand{\xla}[2][]{%
   \ifbool{@display}%
      {\settowidth{\olen}{$\overset{#2}{\longleftarrow}$}%
       \settowidth{\ulen}{$\underset{#1}{\longleftarrow}$}%
       \settowidth{\xlen}{$\xleftarrow[#1]{#2}$}%
       \ifdimgreater{\olen}{\xlen}%
          {\underset{#1}{\overset{#2}{\longleftarrow}}}%
          {\ifdimgreater{\ulen}{\xlen}%
             {\underset{#1}{\overset{#2}{\longleftarrow}}}
             {\xleftarrow[#1]{#2}}}}%
      {\xleftarrow[#1]{#2}}
   }
\newcommand{\isoarrow}{%
   \ifbool{@display}{\overset{\sim}{\longrightarrow}}{\xrightarrow\sim}%
   }
   
\begin{document}

\title[]{Affine Deligne-Lusztig varieties associated with generic Newton points}
\author{Xuhua He}
\address{X.~H., The Institute of Mathematical Sciences and Department of Mathematics, The Chinese University of Hong Kong, Shatin, N.T., Hong Kong SAR, China}
\email{xuhuahe@math.cuhk.edu.hk}

\thanks{}

\keywords{Affine Deligne-Lusztig varieties, Iwahori-Weyl groups, Demazure product}
\subjclass[2010]{20G25, 11G25, 20F55}

\date{\today}

\begin{abstract}
This paper gives an explicit formula of the dimension of affine Deligne-Lusztig varieties associated with generic Newton point in terms of Demazure product of Iwahori-Weyl groups. 
\end{abstract}

\maketitle
{\centering{\em To George Lusztig with admiration\par}}

\section*{Introduction}

Let $F$ be a non-archimedean local field, and $\breve F$ be the completion of its maximal unramified extension. Let $\s$ be the Frobenius automorphism of $\breve F$ over $F$. Let $\bG$ be a connected reductive group over $F$. Let $\breve \CI$ be the standard Iwahori subgroup of $\bG(\breve F)$. Let $w$ be an element in the Iwahori-Weyl group $\tW$ and $b \in \bG(\breve F)$. The affine Deligne-Lusztig variety associated to $(w, b)$ is defined to be $$X_w(b)=\{g \breve \CI \in \bG(\breve F)/\breve \CI; g \i b \s(g) \in \breve \CI w \breve \CI\}.$$ This is a subscheme locally of finite type, of the affine flag variety. It is introduced by Rapoport in \cite{Ra-guide} and plays an important role when studying the special fiber of both Shimura varieties and moduli spaces of Shtukas.

Let $B(\bG)$ be the set of $\s$-conjugacy classes of $\bG(\breve F)$. The set $B(\bG)$ is classified by the Kottwitz map and the Newton map. The set $B(\bG)$ is equipped with a natural partial order by requiring the equality under the Kottwitz map and the dominance order on the associated Newton points. It is easy to see that if $b$ and $b'$ are in the same $\s$-conjugacy class, then $X_w(b)$ and $X_w(b')$ are isomorphic. The set $B(G)_w=\{[b] \in B(G); X_w(b) \neq \emptyset\}$ contains a unique maximal element, which we denote by $[b_w]$. The variety $X_w(b_w)$ is called the affine Deligne-Lusztig varieties associated with generic Newton points. It has been studied in \cite{Mi} and \cite{MV}. 

The main purpose of this note is as follows. 

\begin{theorem}\label{main}
Let $w \in \tW$ and $[b_w]$ be the maximal element in $B(G)_w$. Then $$\dim X_w(b_w)=\ell(w)-\lim_{n \to \infty} \frac{\ell(w^{\ast_\s, n})}{n}.$$ 
\end{theorem}

Here $\ast$ is the Demazure product on $\tW$ and $$w^{\ast_\s, n}=w \ast \s(w) \ast \cdots \ast \s^{n-1}(w)$$ is the $n$-th $\s$-twisted Demazure power of $w$. 

Now we discuss the motivation for this formula and the outline of the proof. 

For simplicity, we only discuss the split group case here. In this case, $\s$ acts trivially on $\tW$ and we may simply write $w^{\ast, n}$ for $w^{\ast_\s, n}$. 

There is a natural bijection between the poset $B(G)$ and the poset of straight conjugacy classes in $\tW$, established in \cite{He14} \& \cite{He16}. Here by definition, an element $x \in \tW$ is straight if $\ell(x^n)=n\ell(x)$ for all $n \in \BN$ and a conjugacy class of $\tW$ is straight if it contains a straight element. 

Let $\CO_w$ be the straight conjugacy class associated to $[b_w]$ and $w'$ be a minimal length element in $C_w$. Then we have $\dim X_w(b_w)=\ell(w)-\ell(w')$. In particular, if $w$ is a straight element, then we may take $w'=w$. In this case, $\dim X_w(b_w)=0$. By definition, for straight element $w$ we have $\ell(w^{\ast, n})=\ell(w^n)=n\ell(w)$. So the statement is obvious in this case. Theorem \ref{main} gives an estimate on the ``non-straightness'' of the element $w$. 

By our assumption, any generic element in the double coset $\breve \CI w \breve \CI$ is $\s$-conjugate to an element in $\breve \CI w' \breve \CI$. Let $g$ be a generic element in $\breve \CI w \breve \CI$ and $g' \in \breve \CI w' \breve \CI$ such that $g$ and $g'$ are $\s$-conjugate by an element $h \in \bG(\breve F)$. We consider the $\s$-twisted power defined by $g^{\s, n}=g \s(g) \cdots \s^{n-1}(g)$. Then for any $n \in \BN$, $g^{\s, n}$ and $(g')^{\s, n}$ are $\s^n$-conjugate by the same element $h$. By the straightness assumption on $\CO_w$, we have $(g')^{\s, n} \in \breve \CI (w')^{n} \breve \CI$. On the other hand, we have $$g^{\s, n} \in (\breve \CI w \breve \CI)  (\breve \CI w \breve \CI) \cdots (\breve \CI w \breve \CI) \subset \overline{\breve \CI w^{\ast,n} \breve \CI}.$$ However, it is not clear if $g^{\s, n} \in \breve \CI w^{\ast, n} \breve \CI$. 

The trick we will use to bypass this difficulty is to apply the technique in \cite{HN3} and to translate the question on the $\s$-conjugation action on $\bG(\breve F)$ to the question on the ordinary conjugation action on a reductive $\bG'$ over $\mathbb C((\e))$. Let $\breve \CI'$ be an Iwahori subgroup of $\bG'(\BC((\e)))$. Using Lusztig's theory of total positivity \cite{Lu-1} and \cite{Lu-2}, we show that for generic element $g \in \breve \CI' w \breve \CI'$, $g^n \in \breve \CI' (w^{\ast, n}) \breve \CI'$. The condition that there exists $h \in \bG'(\mathbb C((\e))$ such that $\breve \CI' (w^{\ast, n}) \breve \CI' \cap h (\breve \CI' (w')^{n} \breve \CI') h \i \neq\emptyset$ for all $n \in \BN$ implies that $\ell(w')=\lim_{n \to \infty} \frac{\ell(w^{\ast, n})}{n}$. This finishes the proof. 

\subsection*{Acknowledgments:} X.H.~is partially supported by a start-up grant and by funds connected with Choh-Ming Chair at CUHK, and by Hong Kong RGC grant 14300220.

\section{Preliminary}

\subsection{Notations}\label{sec:notation} 

Let $\bG$ be a connected reductive group over a non-archimedean local field $F$. Let $\breve F$ be the completion of the maximal unramified extension of $F$ and $\s$ be the Frobenius morphism of $\breve F/F$. The residue field of $F$ is a finite field $\mathbb F_q$ and the residue field of $\breve F$ is the algebraically closed field $\bar{\mathbb F}_q$. We write $\breve G$ for $\bG(\breve F)$. We use the same symbol $\s$ for the induced Frobenius morphism on $\breve G$. 

Let $S$ be a maximal $\breve F$-split torus of $\bG$ defined over $F$, which contains a maximal $F$-split torus. Let $\mathcal A$ be the apartment of $\bG_{\breve F}$ corresponding to $S_{\breve F}$. We fix a $\s$-stable alcove $\mathfrak a$ in $\mathcal A$, and let $\breve \CI \subset \breve G$ be the Iwahori subgroup corresponding to $\mathfrak a$. Then $\breve \CI$ is $\s$-stable.

Let $T$ be the centralizer of $S$ in $\bG$. Then $T$ is a maximal torus. We denote by $N$ the normalizer of $T$ in $\bG$. The \emph{Iwahori--Weyl group} (associated to $S$) is defined as $$\tW= N(\breve F)/T(\breve F) \cap \breve \CI.$$ For any $w \in \tW$, we choose a representative $\dot w$ in $N(\breve F)$. The action $\s$ on $\breve G$ induces a natural action of $\s$ on $\tW$, which we still denote by $\s$. 

We denote by $\ell$ the length function on $\tW$ determined by the base alcove $\mathfrak a$ and denote by $\tilde \BS$ the set of simple reflections in $\tW$. Let $W_{\aff}$ be the subgroup of $\tW$ generated by $\tilde \BS$. Then $W_{\aff}$ is an affine Weyl group. Let $\Omega \subset \tW$ be the subgroup of length-zero elements in $\tW$. Then $$\tW=W_{\aff} \rtimes \Omega.$$ Since the length function is compatible with the $\s$-action, the semi-direct product decomposition $\tW=W_{\aff} \rtimes \Omega$ is also stable under the action of $\s$. 

For any $w \in \tW$, we choose a representative in $N(\breve F)$, which we still denote by $w$. 

%We write $\Gamma$ for $\Gal(\bar F/F)$, and write $\Gamma_0$ for the inertia subgroup of $\Gamma$. We fix a special vertex of the base alcove $\mathfrak a$. Let $W_0=N(\breve F)/T(\breve F)$ be the relative Weyl group. Then we have the splitting $$\tW=X_*(T)_{\G_0} \rtimes W_0=\{t^{\underline \l} w; \underline \l \in X_*(T)_{\G_0}, w \in W_0\}.$$ Note that if $\bG$ is not quasi-split over $F$, then there does not exist a $\s$-stable special vertex in $\mathfrak a$ and thus the splitting $\tW=X_*(T)_{\G_0} \rtimes W_0$ is not $\s$-stable. Let $w_0$ be the longest element in $W_0$. 

%Let $\tilde \BS$ be the set of simple reflections in $\tW$ and $\BS \subset \tilde \BS$ be the set of simple reflections in $W_0$. The action of $\s$ on $\tW$ induces an action on the set $\tilde \BS$, which we still denote by $\s$. However, in general, $\BS$ is not stable under the action of $\s$. 

\subsection{The $\s$-conjugacy classes of $\breve G$}
The $\s$-conjugation action on $\breve G$ is defined by $g \cdot_\s g'=g g' \s(g) \i$ for $g, g' \in \breve G$. Let $B(\bG)$ be the set of $\s$-conjugacy classes on $\breve G$. The classification of the $\s$-conjugacy classes is due to Kottwitz \cite{Ko1} and \cite{Ko2}. Any $\s$-conjugacy class $[b]$ is determined by two invariants: 
\begin{itemize}
	\item The element $\k([b]) \in \pi_1(\bG)_{\s}$; 
	
	\item The Newton point $\nu_b \in \big((X_*(T)_{\Gamma_0, \BQ})^+\big)^{\langle\sigma\rangle}$. 
\end{itemize}

Here $-_\s$ denotes the $\s$-coinvariants, 
$(X_*(T)_{\Gamma_0, \BQ})^+$ denotes the intersection of  $X_*(T)_{\Gamma_0}\otimes \BQ=X_*(T)^{\Gamma_0}\otimes \BQ$ with the set $X_*(T)_\BQ^+$ of dominant elements in $X_*(T)_\BQ$. 

We denote by $\le$ the dominance order on $X_*(T)_\BQ^+$, i.e., for $\nu, \nu' \in X_*(T)_\BQ^+$, $\nu \le \nu'$ if and only if $\nu'-\nu$ is a non-negative (rational) linear combination of positive roots over $\breve F$. The dominance order on $X_*(T)_\BQ^+$ extends to a partial order on $B(\bG)$. Namely, for $[b], [b'] \in B(\bG)$, $[b] \le [b']$ if and only if $\k([b])=\k([b'])$ and $\nu_b \le \nu_{b'}$. 

\subsection{The straight $\s$-conjugacy classes of $\tW$} Let $w \in \tW$ and $n \in \BN$. The {\it $n$-th $\s$-twisted power} of $w$ is defined by $$w^{\s, n}=w \s(w) \cdots \s^{n-1}(w).$$ In the case where $\s$ acts trivally on $\tW$, we have $w^{\s, n}=w^n$ is the ordinary $n$-th power of $w$. 

By definition, an element $w \in \tW$ is called {\it $\s$-straight} if for any $n \in \BN$, $\ell(w^{\s, n})=n \ell(w)$. A $\s$-conjugacy class of $\tW$ is {\it straight} if it contains a $\s$-straight element. Let $B(\tW, \s)_{\text{str}}$ be the set of straight $\s$-conjugacy classes of $\tW$. The following result is proved in \cite[Theorem 3.7]{He14}. 

\begin{theorem}
The map $w \mapsto [w]$ induces a natural bijection $$\Psi: B(\tW, \s)_{\text{str}} \to B(\bG).$$ 
\end{theorem}

\subsection{Affine Deligne-Lusztig varieties} Following \cite{Ra-guide}, we define the affine Deligne-Lusztig variety $X_w(b)$. Let $\mathrm{Fl}=\bG/\breve \CI$ be the affine flag variety. For any $w \in \tW$ and $b \in \bG$, we set $$X_w(b)=\{g \breve \CI \in \mathrm{Fl}; g \i b \s(g) \in \breve \CI w \breve \CI\}.$$ In the equal characteristic, $X_w(b)$ is the set of $\bar{\mathbb F}_q$-points of a scheme. In the equal characteristic, $X_w(b)$ is the set of $\bar{\mathbb F}_q$-points of a perfect scheme (see \cite{BS} and \cite{Zhu}). 

It is easy to see that $X_w(b)$ only depends on $w$ and the $\s$-conjugacy class $[b]$ of $b$. For any $w \in \tW$, we set $$B(\bG)_w=\{[b] \in B(G); X_w(b) \neq \emptyset\}.$$ Let $[b_w]$ be the $\s$-conjugacy class in the (unique) generic point of $\breve \CI w \breve \CI$. Then $[b_w]$ is the unique maximal element in $B(\bG)_w$ with respect to the partial ordering $\le$ on $B(\bG)$ (see \cite[Definition 3.1]{MV}). We choose a representative $b_w \in [b_w]$ and call $X_w(b_w)$ the {\it affine Delinge-Lusztig variety associated with the generic Newton point} of $w$. 

\subsection{Demazure product} Now we recall the Demazure product $\ast$ on $\tW$. By definition, $(W, \ast)$ is a monoid such that $w \ast w'=w w'$ for any $w, w' \in \tW$ if $\ell(w w')=\ell(w)+\ell(w')$ and $s \ast w=w$ for $s \in \tilde \BS$ and $w \in \tW$ if $s w<w$. In other words, $\t \ast w=\t w$ and $s \ast w=\max\{w, sw\}$ for $\t \in \Omega$, $s \in \tilde \BS$ and $w \in \tW$. 

The geometric interpretation of the Demazure product is as follows. For any $w \in \tW$, $\overline{\breve \CI w \breve \CI}=\cup_{w' \le w} \breve \CI w' \breve \CI$ is a closed admissible subset of $\bG$ in the sense of \cite[A.2]{He16}. Then for any $w, w' \in \tW$, we have $$\overline{\breve \CI w \breve \CI} \, \overline{\breve \CI w' \breve \CI}=\overline{\breve \CI (w \ast w') \breve \CI}.$$

Let $w \in \tW$ and $n \in \BN$. The {\it $n$-th $\s$-twisted Demazure power} of $w$ is defined by $$w^{\ast_\s, n}=w \ast \s(w) \ast \cdots \ast \s^{n-1}(w).$$ 

\subsection{Minimal length elements}  For $w, w' \in  \tW$ and $s \in \tilde \BS$, we write $w \xrightarrow{s}_\s w'$ if $w'=s w \s(s)$ and $\ell(w') \le \ell(w)$. We write $w \rightarrow_\s w'$ if there is a sequence $w=w_1, w_2, \dotsc, w_n=w'$ in $\tW$ such that for each $2 \leq k \leq n $ we have $w_{k-1} \xrightarrow{s_k}_\s w_k$ for some $s_k \in \tilde \BS$. We write $w \approx_\sigma w'$ if $w \to_\sigma w'$ and $w' \to_\sigma w$. We write $w \tilde \approx_\s w'$ if $w \approx_\s \t w' \s(\t) \i$ for some $\t \in \Omega$. For any $\s$-conjugacy class $\CO$, we write $\ell(\CO)=\ell(x)$ for any minimal length element $x$ of $\CO$. %By ?, the minimal length elements in a straight $\s$-conjugacy class are $\s$-straight. 

The following result is proved in \cite[Theorem A]{HN2}. 

\begin{theorem}\label{thm:min}
Let $w \in \tW$. Then there exists a minimal length element $w'$ in the same $\s$-conjugacy class of $w$ such that $w \to_\s w'$. 
\end{theorem}

\section{The generic $\s$-conjugacy class}

In this section, we study the $\s$-conjugacy class $[b_w]$ in more detail. 

\subsection{Via the Bruhat order} Let $w \in \tW$. The generic $\s$-conjugacy class $[b_w]$ in $\breve \CI w \breve \CI$ is first studied by Viehmann in \cite{Vi}. The following result is proved in \cite[Corollary 5.6]{Vi}

\begin{proposition}\label{lem:gen-1}
Let $w \in \tW$. Then the set $$\{[w']; w' \le w\} \subset B(\bG)$$ contains a unique maximal element and this maximal element equals $[b_w]$. 
\end{proposition}

A more explicit description of the generic Newton point $\nu_{b_w}$ is obtained by Mili\'cevi\'c \cite[Theorem 3.2]{Mi} for split group $\bG$ and sufficiently large $w$. 

\subsection{Via the partial order on $B(\tW, \s)_{\text{str}}$} Let $w \in \tW$ and $\CO \in B(\tW, \s)_{\text{str}}$. We write $\CO \preceq_\s w$ if there exists a minimal length element $w' \in \CO$ such that $w' \le w$ with respect to the Bruhat order $\le$ of $\tW$. Let $\CO, \CO' \in B(\tW, \s)_{\text{str}}$. We write $\CO' \preceq_\s \CO$ if $\CO' \preceq_\s w$ for some minimal length element $w$ of $\CO$. By \cite[\S 3.2]{He16}, if $\CO' \preceq_\s \CO$, then $\CO' \preceq_\s w'$ for any minimal length element $w'$ of $\CO$. Hence $\preceq_\s$ is a partial order on $B(\tW, \s)_{\text{str}}$. 

It is proved in \cite[Theorem B]{He16} that 

\begin{theorem}\label{thm:two-posets}
The partial order $\preceq_\s$ on $B(\tW, \s)_{\text{str}}$ coincides with the partial order $\le$ on $B(\bG)$ via the bijection map $\Psi: B(\tW, \s)_{\text{str}} \to B(\bG)$. 
\end{theorem}

Now we show that 

\begin{proposition}\label{prop:gen-2}
Let $w \in \tW$. Then the set $\{\CO \in B(\tW, \s)_{\text{str}}; \CO \preceq_\s w\}$ contains a unique element $\CO_w$ with respect to the partial order $\preceq_\s$. Moreover, $\Psi(\CO_w)=[b_w]$. 
\end{proposition}

\begin{proof}
By \cite[\S 2.7]{He16}, $\cup_{w'\le w} [w']=\cup_{\CO \preceq_\s w} \Psi(\CO)$. By Proposition \ref{lem:gen-1}, the set $$\{[w']; w' \le w\}=\{\Psi(\CO) \in B(\tW, \s)_{\text{str}}; \CO \preceq_\s w\}$$ contains a unique maximal element, which is $[b_w]$. Let $\CO_w=\Psi \i([b_w])$. By Theorem \ref{thm:two-posets}, $\CO_w$ is the unique maximal element of $\{\CO \in B(\tW, \s)_{\text{str}}; \CO \preceq_\s w\}$. 
\end{proof}

\subsection{An algorithm}\label{sec:alg} We provide an algorithm to compute $\CO_w$. We argue by induction on $\ell(w)$. By \cite[\S 2.7]{He16}, if $w$ is a minimal length element in its $\s$-conjugacy class, then $\CO_w$ is the straight $\s$-conjugacy class of $\tW$ that corresponds to $[w] \in B(\bG)$. 

If $w$ is not a minimal length element in its $\s$-conjugacy class, then by Theorem \ref{thm:min}, there exists $w' \in \tW$ and a simple reflection $s$ such that $w' \approx_\s w$ and $s w' \s(s)<w'$. By \cite[Proposition 2.4]{He16}, $\CO \preceq_\s w$ if and only if $\CO \preceq_\s w'$. Let $\CO$ be a straight $\s$-conjugacy class with $\CO \preceq_\s w'$. Then there exists a minimal length element $w_1$ of $\CO$ with $w_1 \le w'$. If $s w_1>w_1$, then $w_1 \le \min\{w', s w'\}=s w'$. If $s w_1<w_1$, then $\ell(s w_1 \s(s)) \le \ell(s w_1)+1=\ell(w_1)$ and $s w_1 \s(s)$ is also a minimal length element in $\CO$. Moreover, we have $s w_1 \le s w'$ and $s w_1 \s(s) \le \max\{s w', s w' \s(s)\}=s w'$. In either case, $\CO \preceq_\s {s w'}$. By inductive hypothesis on $s w'$, $\CO_w=\CO_{w'}=\CO_{s w'}$ is the unique maximal element in $\{\CO; \CO \preceq_\s w\}=\{\CO; \CO \preceq_\s s w'\}$. 

\subsection{Via the $0$-Hecke algebras} Let $H_0$ be the {\it $0$-Hecke algebra} of $\tW$. It is a $\mathbb C$-algebra generated by $\{t_w; w \in \tW\}$ subject to the relations 
\begin{itemize}
    \item $t_w t_{w'}=t_{w w'}$ for any $w, w' \in \tW$ with $\ell(w w')=\ell(w)+\ell(w')$. 
    
    \item $t_s^2=-t_s$ for any $s \in \tilde \BS$. 
\end{itemize}

The automorphism $\s$ on $\tW$ induces a natural algebra homomorphism on $H_0$, which we still denote by $\s$. For any $h, h' \in H_0$, the $\s$-commutator of $h$ and $h'$ is defined by $[h, h']_\s=h h'-h'\s(h)$. The {\it $\s$-commutator} $[H_0, H_0]_\s$ of $H_0$ is by definition the subspace of $H_0$ spanned by $[h, h']_\s$ for all $h, h' \in H_0$. The $\s$-cocenter of $H_0$ is defined to be $\bar H_{0 , \s}=H_0/[H_0, H_0]_\s$. 

Let $\tW_{\s, \min}$ be the set of elements in $\tW$ which are of minimal length in their $\s$-conjugacy classes. It is easy to see that if $w \in \tW_{\s, \min}$ and $w' \tilde \approx_\s w$, then $w' \in \tW_{\s, \min}$. Let $\tW_{\s, \min}/\tilde \approx_\s$ be the set of $\tilde \approx_\s$-equivalence classes in $\tW_{\s, \min}$. 

By \cite[Proposition 2.1]{He15}, if $w \tilde \approx_\s w'$, then $t_w$ and $t_{w'}$ have the same image in $\bar H_{0, \s}$. For any $\Sigma \in \tW_{\s, \min}/\tilde \approx_\s$, we write $t_{\Sigma}$ for the image of $t_w$ in $\bar H_{0, \s}$ for any $w \in \Sigma$. 

We have the following result. 

\begin{proposition}
(1) The set $\{t_{\Sigma}\}_{\Sigma \in \tW_{\s, \min}/\tilde \approx_\s}$ is a $\mathbb C$-basis of $\bar H_{0, \s}$. 

(2) For any $w \in \tW$, there exists a unique $\Sigma_w \in \tW_{\s, \min}/\tilde \approx_\s$ such that the image of $t_w$ in $\bar H_{0, \s}$ equals $\pm t_{\Sigma_w}$. 
\end{proposition}

This result is proved for $0$-Hecke algebras of finite Weyl groups in \cite[Proposition 5.1 \& Proposition 6.2]{He15}. The proof for the $0$-Hecke algebras of Iwahori-Weyl groups is the same. 

Now we give another description of $[b_w]$. 

\begin{proposition}\label{prop:gen-3}
Let $w \in \tW$. Then $[b_w]=\Psi(\Sigma_w)$. 
\end{proposition}

\begin{remark}
It is worth pointing out that in general, $\Sigma_w$ is different from $\CO_w$ in Proposition \ref{prop:gen-2}. 
\end{remark}

\begin{proof}
The argument is similar to \S\ref{sec:alg}. We argue by induction on $\ell(w)$. By \cite[\S 2.7]{He16}, if $w$ is a minimal length element in its $\s$-conjugacy class and $t_{\Sigma_w}$ is the $\tilde \approx_\s$-equivalence class of $w$, then $\Psi(\Sigma_w)=[w]=[b_w]$. 

If $w$ is not a minimal length element in its $\s$-conjugacy class, then by Theorem \ref{thm:min}, there exists $w' \in \tW$ and a simple reflection $s$ such that $w' \approx_\s w$ and $s w' \s(s)<w'$. We have $$t_w \equiv t_{w'}=t_s t_{s w'} \equiv t_{s w'} t_{\s(s)}=-t_{s w'} \mod [H_0, H_0]_\s.$$ Therefore $\Sigma_w=\Sigma_{w'}=\Sigma_{s w'}$. By \S\ref{sec:alg}, $\CO_w=\CO_{w'}=\CO_{s w'}$. Now the statement follows from induction hypothesis on $s w'$.  
\end{proof}

\section{Passing from non-archimedean local fields to $\mathbb C((\e))$}

\subsection{Dimension formula} The following result follows from \cite[Theorem 2.23]{He-CDM}. 

\begin{lemma}
Let $w \in \tW$. Then $$\dim X_w(b_w)=\ell(w)-\ell(\CO_w).$$ 
\end{lemma}

Note that the definition of $\CO_w$ only depends on the triple $(\tW, \s, w)$ and is independent of the reductive $\bG$ over $F$. This allows us to reduce the calculation of the dimension of affine Deligne-Lusztig varieties to the calculation of the length of certain elements in the Iwahori-Weyl group. And the latter problem will be translated to a problem on reductive groups over $\mathbb C((\e))$. This is what we will do in this section. 

\subsection{Reduction to $W_{\aff}$}\label{sec:reduction}
Let $w=x \t$ for $x \in W_{\aff}$ and $\t \in \Omega$. We write $\th=\Ad(\t) \circ \s \in \Aut(W_\aff)$. Define the map $$\iota: W_\aff \to \tW, \quad x' \mapsto x' \t.$$ For any $x' \in W_\aff$ and $n \in \BN$, we have $\ell(\iota(x')^{\s, n})=\ell((x')^{\th, n})$. Thus $x'$ is $\th$-straight if and only if $\iota(x')$ is $\s$-straight. It is also easy to see that if $x_1, x_2$ are in the same $\th$-conjugacy class of $W_\aff$, then $\iota(x_1), \iota(x_2)$ are in the same $\s$-conjugacy class of $\tW$. The map $\iota$ induces a map $B(W_\aff, \th)_{\text{str}} \to B(\tW, \s)_{\text{str}}$, which we still denote by $\iota$. 

By Proposition \ref{prop:gen-2}, there exists a unique maximal element $\CO_x$ in $$\{\CO' \in B(W_\aff, \th)_{\text{str}}; \CO' \preceq_\th x\}.$$ By definition, $\iota(\CO_x) \preceq_\s \iota(x)=w$ and it is a maximal element in $$\{\CO \in B(\tW, \s)_{\text{str}}; \CO \preceq_\s w\}.$$ Thus $\iota(\CO_x)=\CO_w$. We have $\ell(w)-\ell(\CO_w)=\ell(x)-\ell(\CO_x)$. 

Thus to prove Theorem \ref{main}, it remains to show that for any diagram automorphism $\th$ of $W_\aff$, we have \[\tag{*} \ell(\CO_x)=\lim_{n \to x} \frac{\ell(x^{\th, n})}{n}.\]

\subsection{The group $\bG'$ over $\mathbb C((\e))$} Let $\bG'$ be a connected semisimple group split over $\BC((\e))$ whose Iwahori-Weyl group is isomorphic to $W_\aff$. We write $\breve G'$ for $\bG'(\BC((\e)))$. Let $T'$ be a split maximal torus of $\bG'$ and $B' \supset T'$ be a Borel subgroup. Let $$\breve \CI'=\{g \in G(\BC[[\e]]); g \mid_{\e \mapsto 0} \in B'(\BC)\}$$ be an Iwahori subgroup of $\breve G'$. We have the decomposition $\breve G'=\sqcup_{x \in W_\aff} \breve \CI' x \breve \CI'$. 

The diagram automorphism $\th$ on $W_\aff$ can be lifted to a diagram automorphism on $\breve G'$, which we still denote by $\th$. We consider the $\th$-conjugation  action $\cdot_\th$ on $\breve G'$ here. 

Let $\CO \in B(W_\aff, \th)_{\text{str}}$. By \cite[\S3.2]{HN3}, $\breve G' \cdot_\th \breve \CI' x \breve \CI'=\breve G' \cdot_\th \breve \CI' x' \breve \CI'$ for any minimal length elements $x, x' \in \CO$. We write $[\CO]=\breve G' \cdot_\th \breve \CI' x \breve \CI'$ for any minimal length element $x \in \CO$. We have 

\begin{theorem}
$\breve G'=\bigsqcup_{\CO \in B(W_\aff, \th)_{\text{str}}} [\CO]$.
\end{theorem}

It is proved in \cite[Theorem 3.2]{HN3} for reductive groups over $\kk((\e))$, where $\kk$ is an algebraically closed field of positive characteristic. The same proof works over $\BC((\e))$. 

As a variation of Proposition \ref{prop:gen-2}, $\breve G \cdot_\s \overline{\breve \CI w \breve \CI}=\bigsqcup_{\CO \preceq_\s \CO_w} \Psi(\CO_w)$. Similarly we have the following result. 

\begin{proposition}
Let $x \in W_\aff$. Then $$\breve G' \cdot_\th \overline{\breve \CI' x \breve \CI'}=\bigsqcup_{\CO \preceq_\th \CO_x} [\CO].$$
\end{proposition}

\section{Generic elements coming from total positivity}
Let $x \in W_\aff$. Then for any $g \in \breve \CI' x \breve \CI'$, we have $g \in \sqcup_{\CO \preceq_\th \CO_x} [\CO]$ and $g^{\th, n} \in \sqcup_{x' \le x^{\ast_\th, n}} \breve \CI' x' \breve \CI'$ for all $n$. The advantage of working with $\BC((\e))$ instead of $\breve F$ is that we may prove the following result on the generic elements of $\breve \CI' x \breve \CI'$. 

\begin{proposition}\label{prop: generic}
Let $x \in W_\aff$. Then there exists $g \in \breve \CI' x \breve \CI'$ such that $g \in [\CO_x]$ and for any $n \in \BN$, $g^{\th, n} \in \breve \CI' x^{\ast_\th, n} \breve \CI'$. 
\end{proposition}

\begin{proof}
The idea is to use Lusztig's theory of total positivity. 
Recall that $\tilde \BS$ is the set of positive affine simple roots of $\breve G'$. For any $i \in \tilde \BS$, let $\a_i$ be the corresponding affine simple root and $U_{\a_i} \subset \breve G'$ be the corresponding affine root subgroup. Let $\{x_i: \BG_m \to U_{\a_i}; i \in \tilde \BS\}$ be an affine pinning of $\bG'$ (see \cite[\S 5.3]{GH}). Since $\th$ is a diagram automorphism of $\bG'$, we may choose $\{x_i\}$ to be $\th$-stable. Let $U_{-\a_i}$ be the affine root subgroup correspond to $-\a_i$. Let $y_i: \BG_m \to U_{-\a_i}$ be the isomorphism such that $x_i(1) y_i(-1) x_i(1) \in N(T')$\footnote{This normalization of $y_i$ differs from \cite{GH} but is consistent with the normalization used in Lusztig's theory of total positivity.}.

Let $y \in W_\aff$ and $y=s_{i_1} \cdots s_{i_k}$ be a reduced expression of $y$. Set $$U^-_y=\{y_{i_1} (a_1) \cdots y_{i_k}(a_k); a_1, \ldots, a_k \in \BR_{>0}\}.$$

By \cite[\S2.5]{Lu-2}, $U^-_y$ is independent of the choices of reduced expressions of $y$. Moreover, since $y_i(a) \in \breve \CI' s_i \breve \CI'$ for any $i \in \tilde \BS$ and $a \neq 0$, we have $U^-_y \subset \breve \CI' y \breve \CI'$. 

By \cite[\S2.11]{Lu-2}, for any $y_1, y_2 \in W_\aff$, we have $U^-_{y_1} U^-_{y_2}=U^-_{y_1 \ast y_2}$.

In particular, for any $g \in U^-_x$ and $n \in \BN$, we have $$g^{\th, n} \in U^-_x \cdot U^-_{\th(x)} \cdots U^-_{\th^{n-1}(x)} \subset U^-_{x^{\ast_\th, n}} \subset \breve \CI' x^{\ast_\th, n} \breve \CI'.$$

We show that 

(a) If $s y<y$ for some simple reflection $s$, then any element in $U^-_y$ is $\th$-conjugate to an element in $U^-_{(s y) \ast \th(s)}$. 

By definition, there exists a reduced expression of $y$ with $y=s_{i_1} \cdots s_{i_k}$ and $s_{i_1}=s$. Thus $y_{i_1} (a_1) \cdots y_{i_k}(a_k)$ is $\th$-conjugate to $$y_{i_2} (a_2) \cdots y_{i_k}(a_k) \th(y_{i_1}(a_1))=y_{i_2} (a_2) \cdots y_{i_k}(a_k) y_{\th(i_1)}(a_1).$$ If $a_1, \ldots, a_k>0$, then $y_{i_2} (a_2) \cdots y_{i_k}(a_k) y_{\th(i_1)}(a_1) \in U^-_{s y} U^-_{\th(s)}=U^-_{(s y) \ast \th(s)}$. (a) is proved. 

Now we show that $g \in [\CO_x]$. We argue by induction on $\ell(x)$. 

If $x$ is a minimal length element in its $\th$-conjugacy class of $W_\aff$, then by the reduction argument in \cite[Lemma 3.1]{He14}, $g \in \breve \CI' x \breve \CI' \subset [\CO_x]$. If $x$ is not a minimal length element in its $\th$-conjugacy class of $W_\aff$, then by Theorem \ref{thm:min}, there exists $x' \in W_{\aff}$ and a simple reflection $s$ such that $x' \approx_\th x$ and $s x' \th(s)<x'$. We have $s x'<x'$ and $(s x') \ast \th(s)=s x'$. By (a), any element in $U^-_x$ is $\th$-conjugate to an element in $U^-_{x'}$ and any element in $U^-_{x'}$ is $\th$-conjugate to an element in $U^-_{s x'}$. By inductive hypothesis on $s x'$, we have $U^-_{s x'} \subset [\CO_{s x'}]$. By \S\ref{sec:alg}, $\CO_{s x'}=\CO_{x'}=\CO_x$. 

This finishes the proof. 
\end{proof}

\subsection{Proof of Theorem \ref{main}} Now we prove Theorem \ref{main}. As explained in \S\ref{sec:reduction}, it suffices to prove \S\ref{sec:reduction}(*).  By Proposition \ref{prop: generic}, there exists $g \in \breve \CI' x \breve \CI'$ such that $g \in [\CO_x]$ and for any $n \in \BN$, $g^{\th, n} \in \breve \CI' x^{\ast_\th, n} \breve \CI'$. Let $x'$ be a minimal length element in $\CO_x$. Then $x'$ is a $\th$-straight element. Since $g \in [\CO_x]$, there exists $h \in \breve G'$ and $g' \in \breve \CI' x' \breve \CI'$ such that $g'=h g \th(h) \i$. 

Since $x'$ is $\th$-straight, we have $$(g')^{\th, n} \in (\breve \CI' x' \breve \CI') (\breve \CI' \th(x') \breve \CI') \cdots (\breve \CI' \th^{n-1}(x') \breve \CI')=\breve \CI' (x')^{\th, n} \breve \CI'.$$

On the other hand, $$(g')^{\th, n}=h g^{\th, n} \th^n(h) \i \in h (\breve \CI' x^{\ast_\th, n} \breve \CI') \th^n(h) \i.$$

We have $h \in \breve \CI' y \breve \CI'$ for some $y \in W_\aff$. Let $N_0=\ell(y)$. Then $\th^n(h) \in \breve \CI' \th^n(y) \breve \CI'$ and $\ell(\th^n(y))=N_0$. 

Note that \begin{align*} h (\breve \CI' x^{\ast_\th, n} \breve \CI') \th^n(h) \i & \subset (\breve \CI' y \breve \CI') (\breve \CI' x^{\ast_\th, n} \breve \CI') (\breve \CI' \th^n(y) \breve \CI') \\ & \subset \bigsqcup_{z \in W_\aff; \ell(x^{\ast_\th, n})-2 N_0 \le \ell(z) \le \ell(x^{\ast_\th, n})+2 N_0} \breve \CI' z \breve \CI'.
\end{align*}

Therefore $$\ell(x^{\ast_\th, n})-2 N_0 \le \ell((x')^{\th, n}) \le \ell(x^{\ast_\th, n})+2 N_0$$ for all $n \in \BN$. Since $x'$ is $\th$-straight, $\ell((x')^{\th, n})=n \ell(x')$. Thus $$\frac{\ell(x^{\ast_\th, n})}{n}-\frac{2 N_0}{n} \le \ell(x') \le \frac{\ell(x^{\ast_\th, n})}{n}+\frac{2 N_0}{n}.$$ Therefore $\ell(\CO_x)=\ell(x')=\lim_{n \to \infty} \frac{\ell(x^{\ast_\th, n})}{n}$. The proof is finished.

\end{document}